\documentclass[12pt]{article}
\usepackage{amsmath}
\usepackage{graphicx,psfrag,epsf}
\usepackage{enumerate}
\usepackage{natbib}
\usepackage{amsmath,amsthm,amsfonts,epsfig,amssymb,natbib,eucal,eufrak}

\usepackage{booktabs}
\usepackage{multirow}
\usepackage{siunitx}
\usepackage{qtree}
\usepackage{hyperref}

\usepackage{float}
\restylefloat{table}
\usepackage{color}

\usepackage{pdfpages}
\usepackage{wrapfig}

\usepackage{comment}

\usepackage{tikz}

\newtheorem{thm}{Theorem}
\newtheorem{lem}{Lemma}

\newtheorem{ex}{Example}
\newtheorem{result}{Result}
\newtheorem{com}{Comment}
\newtheorem{remark}{Remark}

\newtheorem{cor}{Corollary}

\newcommand{\blind}{0}

\begin{document}

\def\spacingset#1{\renewcommand{\baselinestretch}%
{#1}\small\normalsize} \spacingset{1}


\if0\blind
{
  \title{\bf On negative dependence inequalities and maximal scores in round-robin tournaments}

  \author{
  Yaakov Malinovsky
    \thanks{email: yaakovm@umbc.edu}
   \\
    Department of Mathematics and Statistics\\ University of Maryland, Baltimore County, Baltimore, MD 21250, USA\\
    \\
  John W. Moon
    \thanks{email: jwmoon@ualberta.ca}
   \\
    Department of Mathematical and Statistical Sciences\\ University of Alberta, Edmonton, AB T6G 2G1, Canada\\
}
  \maketitle
} \fi

\if1\blind
{
  \bigskip
  \bigskip
  \bigskip
  \begin{center}
    {\LARGE\bf Title}
\end{center}
  \medskip
} \fi

\begin{abstract}
We extend Huber's (1963) inequality for the joint distribution function of negative
dependent scores in round-robin tournaments. As a byproduct, this extension implies convergence in
probability of the maximal score in round-robin tournaments in a more general setting.

\end{abstract}

\noindent%

\noindent%

\spacingset{1.40} 
\section{Introduction and Background}
In a classical round-robin tournament, each of $n$ players wins or loses against each of the other $n-1$ players \citep{Moon2013}.
Denote by $X_{ij}$ the score of player $i$ after the game with player $j, j\neq i$.
We assume that all  ${\displaystyle {n \choose 2}}$ pairs of scores $\left(X_{12}, X_{21}\right),\ldots,\left(X_{1n}, X_{n1}\right),\ldots,$ $\left(X_{n-1,n}, X_{n,n-1}\right)$ are independent.
Let $s_i=\sum_{j=1, j\neq i}^{n}{X_{ij}}$ be the score of player $i$ $(i=1,\ldots,n)$ after playing with all $n-1$ opponents.
We use a standard notation and denote by $s_{(1)}\leq s_{(2)}\leq \ldots \leq s_{(n)}$ the order statistics of the random variables
$s_1,s_2,\ldots,s_n$; and we denote by $s_{1}^{*}, s_{2}^{*},\ldots,s_{n}^{*}$
normalized scores (zero expectation and unit variance)
with corresponding order statistics $s_{(1)}^{*}, s_{(2)}^{*},\ldots,s_{(n)}^{*}$.

Measuring players strengths in chess tournaments by modeling paired comparisons of strength has a long history and appeared in \cite{Z1929}.
Zemerlo's model is given by $P(\text{player}\,\,\,i\,\,\, defeats\,\,\,\, \text{player}\,\,\,\,j)=P(X_{ij}=1)={\displaystyle \frac{u_i}{u_i+u_j}},$ where $u_i$ and $u_j$ are unknown strengths of players $i$ and $j$, respectively, and for $i\neq j$, $X_{ij}+X_{ji}=1, X_{ij}\in \left\{0,1\right\}$.
\cite{Z1929} used the maximum likelihood (ML) method to estimate the parameters $\left\{u_i\right\}$.
This approach was rediscovered by \cite{BT1952} and \cite{F1957} and is usually referred to as the Bradley-Terry model; see, e.g. \cite{D1988}.
Interesting historical comments related to Zemerlo's model
can be found in the article "Comments on Zermelo (1929)" of the book by \cite{DE2001} and in \cite{G2013}.

\cite{SY1999} estimate $\left\{ u_i\right \}$ based on the data $X_{ij}, 1\leq i<j\leq n$, and proved the consistency and  asymptotic normality of the ML estimators.
\cite{CDA2011} investigated maximum likelihood estimates pertaining to the degree sequences of random graphs generated by what the authors refer to as a "close cousin" of the Bradley-Terry model.
Further, \cite{CDL2017} investigated the asymptotic probability that the best player wins, assuming that the strengths of the players are random variables.

Let $r_n$ denote the probability  that an ordinary round-robin tournament with $n$ labelled vertices has a unique vertex with maximum score, assuming all the ${\displaystyle 2^{\binom{n}{2}}}$ such tournaments are equally likely. \cite{E2013} (Section nine) gave the values $r_4 = .5, r_5 = .586, r_6 =.627, r_7 = .581$, and $r _8 = .634$ ({\it We remark that it follows from Table 1 in \cite{D1959} that $r_8 = 160,241,152/2^{28} = .596\ldots$}). Epstein also stated that as $n$ increases indefinitely, $r_n$ approaches unity. However, we are not aware of a proof of the final conclusion.

Material on round-robin tournaments can be found in \cite{HM1966}, \cite{Moon2013}, and \cite{R2014}.
\cite{L1953} has given necessary and sufficient conditions for a set of integers to be the score sequence of some tournament.
Landau was interested in animal behavior, and his work grew out of dealing with the pecking orders of chickens.
Landau's theorem has been reproved and generalized by a number of authors in a variety of ways; see, for example,
\cite{Moon2013} (Section 22), the survey paper \cite{GR1999}, and, more recently, \cite{HMR2011}, and \cite{BF2015}.

\cite{H1963} was concerned with the asymptotic behavior of the
highest score in a paired comparison experiment when the number $n$ of treatments (players) is very
large. He assumed that the players are all of equal strength, except for a single 'outlier,' which will be preferred with probability $p>1/2$ when compared with any other player. Each pair of players is compared exactly once and no ties are permitted. Huber proved that the probability that the outlier has the maximum score tends to $1$ for all fixed $p>1/2$ as $n$ tends to infinity.
A byproduct of Huber's work when $p=1/2$ is the following result, which he gave as a Corollary.

\begin{result}[\cite{H1963}]
\label{eq:General}
If $X_{ij}\in \left\{0,1\right\}, X_{ij}+X_{ji}=1, p_{ij}=P(X_{ij}=1)=\frac{1}{2}$, and
if $n\rightarrow \infty$, then $s_{(n)}^{*}-\sqrt{2\log(n-1)}\rightarrow 0$ in probability.
\end{result}

A key step in Huber's approach was an inequality for the joint cumulative distribution function of the negatively dependent scores $s_1,\ldots,s_n$, where $X_{ij}+X_{ji}=1$; extending this inequality permits one to estimate the maximal scores in more general tournament settings. A verbatim statement of Huber's inequality is given below.

\begin{lem}[\cite{H1963}]
\label{eq:Lemma}
For any probability matrix $(p_{ij})$ and any numbers $(k_1,\ldots,k_m)$, $m\leq n$, the joint cumulative distribution function of the scores $s_1,\ldots,s_m$ satisfies
\begin{equation}
\label{eq:iH}
P\left(s_1<k_1,\ldots,s_m<k_m\right)\leq P\left(s_1<k_1\right)\cdots P\left(s_m<k_m\right).
\end{equation}
\end{lem}

A similar inequality for negatively correlated normal random variables appears in \cite{S1962}, for the multinomial random variables in \cite{M1968}, and for other multivariate discrete distributions in \cite{JP1975}.

In particular, \cite{JP1975} showed that if $A$ denotes a measurable event and ${\displaystyle P(A|X=x)}$ is a well-defined nondecreasing function of $x$, then $P\left(A, X\leq a\right)\leq P\left(A\right)P\left(X\leq a\right)$ for every $a$.
Their proof is based on Chebyshev's order inequality.
In order to use their result to prove \eqref{eq:Lemma} in the tournaments setting we would need to verify that
$P\left(s_{j+1}\leq k_{j+1},\ldots,s_n<k_n|s_j=k\right)$ is a nondecreasing function of $k$ for $j=1,\ldots,n-1$,
and it would be a hard task.

\cite{L1966} called two random variables {\it negative quadrant dependent}
if they satisfy \eqref{eq:iH}(see also \cite{N2006}). The random variables which satisfy \eqref{eq:iH} are known as {\it negative lower orthant dependent (NLOD)} and were investigated in \citep{JP1983} and references therein. All this is closely related to the property of {\it negatively associated (NA)} random variables $X_1, X_2,\ldots,X_n$, where for every pair of disjoint subsets $A_1, A_2$ of $\left\{1,2,\ldots,n\right\}$,
${\displaystyle
Cov\left\{f(X_i, i\in A_1), g(X_j, j\in A_ 2)\right\}\leq 0,
}
$
for all nondecreasing functions $f, g$.
NA implies NLOD, but not vise versa \citep{JP1983}.
In addition, \cite{JP1983} show that negatively correlated normal random variables, which are NLOD \citep{S1962}, also are NA.

A recent result was published in July 2021 \citep{R2021}, which is closely related to \eqref{eq:Lemma}.
Ross considered a tournament model where $X_{ij}\sim Bin (n_{ij}, p_{ij})$.
The binomial distribution is log-concave, i.e., for all $u\geq 1$, $(p(u))^2\geq p(u-1)p(u+1),$ where $p(u)=P(X_{ij}=u)$
\citep{JG2006}.
\cite{R2021} (Proposition 1) used a theorem from \cite{E1965} on log-concave distributions to show that in the model
he was studying, ${\displaystyle s_{-i}\,\big |\,s_i=k}$
is stochastically decreasing in $k$, where ${\displaystyle s_{-i}=\left(s_1,\ldots,s_{i-1},s_{i+1},\ldots,s_n\right)}$;
that is, ${\displaystyle E\left(\Phi(s_{i-1})| s_i=k\right)}$ is a non-increasing function of $k$ for any
a real measurable function $\displaystyle{\Phi(x_1,\ldots,x_{n-1})}$
on Euclidean $(n-1)$-space  which is non-decreasing in each of its arguments.
\cite{R2021} (Corollary 2) then deduced that ${\displaystyle s_{-i}\,\big |\,s_i\geq k}$
is stochastically smaller than $s_{-i}$. This implies
that ${\displaystyle I_{-i}\,\big |\,I_i=1}$ is stochastically smaller than $I_{-i}$, where $I_i$ is the indicator function of the event that $s_i>k$, and ${\displaystyle I_{-i}=\left(I_1,\ldots,I_{i-1}, I_{i+1},\ldots,I_{n}\right)}$. Then, it follows from (Proposition 2, \cite{R2016}) that ${\displaystyle P\left(\sum_{i=1}^{n}I_i=0\right)\leq \prod_{i=1}^{n}P\left(I_i=0\right) },$ which is equivalent to $
P\left(s_1\leq k,\ldots,s_n\leq k\right)\leq P\left(s_1\leq k\right)\cdots P\left(s_n\leq k\right)
$.
\cite{JP1983} (Theorem 2.8) gave another result that made use of the work of \cite{E1965}, namely, that
if $X_1,\ldots,X_n$ are independent random variables with log-concave densities then the joint conditional distribution
of $X_1,\ldots,X_n$ given $\sum_{i=1}^{n}X_i$ is NA. Using this result, they show that the multinomial random variables are NA and therefore NLOD.

\cite{M2021a, M2021b} considered the chess round-robin tournament model (see Example \ref{eq:Chess} below)
and found an asymptotic distribution of  $s_{(i)}^{*}, i=1,\ldots,n$.

In this work, we extend Huber's lemma to a large class of discrete distributions of $X_{ij}$ and, as a byproduct, show that this extension implies convergence in probability of the normalized maximal score for generalizations of round-robin tournaments.

\section{Main Results}
Suppose that $n (\geq 2)$ players participate in a generalized round--robin tournament and that each player is compared with each of the other $n - 1$ players (one or more times); and that as a result of the comparison(s) between players $i$ and $j$, these players receive $X_{ij}$ and $X_{ji}$ points, respectively, where $X_{ij}$ and $X_{ji}$ range over the integers $0, 1,\ldots, m$, for some fixed positive integer $m$, and $X_{ij} + X_{ji} = m$. We further assume that for each given ordered pair of distinct integers $(i, j)$, $1\leq i, j\leq n$, there exist nonnegative numbers $p_0,\ldots,p_m$
such that $p_0+\cdots+p_m=1$ and
\begin{equation}
p_u=P(X_{ij} = u)=P(X_{ij} = u, X_{ji}=m - u)=P( X_{ji} = m - u)
\tag{L}\label{eq:L}
\end{equation}
for $0\leq u \leq m$.

Notice that it follows from assumption \eqref{eq:L} and the fact that $X_{ij}+X_{ji}=m$ that $E(X_{ij}) + E(X_{ji})=m$ and $Var(X_{ij})=Var(X_{ji})$.

\begin{thm}
\label{eq:Lemma2}
If the probabilities associated with a generalized tournament satisfy condition \eqref{eq:L}, then for any fixed nonnegative integers $k_1,\ldots, k_n$, the joint distribution function $F(k_1,\ldots,k_n)$ of the
scores $s_1,\ldots,s_n$ satisfies the relation
\begin{equation}
\label{eq:M}
F\left(k_1,\ldots, k_n\right)=P\left(s_1\leq k_1,\cdots, s_n\leq k_n\right)\leq P\left(s_1\leq k_1\right)\cdots P\left(s_n \leq k_n\right).
\end{equation}

\end{thm}


\begin{proof}
Any two particular scores $s_1$ and $s_2$, say, can be rewritten as
$s_1 =  s^{'}_1 + X_{12}$ and  $s_2 = s^{'}_2 + X_{21}$, where
$s^{'}_1 = \sum_{j \ne 1,2}X_{1j}$ and $s^{'}_2=\sum_{j \ne 2,1}X_{2j}$.
So the expression for $F(k_1,\ldots,k_n)$ can be rewritten as
\begin{equation}
\label{eq:M2}
F:=F\left(k_1,\ldots, k_n\right)=P(s'_1 + X_{12} \leq k_1, s'_2 + X_{21}\leq k_2, s_3\leq k_3,\ldots, s_n\leq k_n).
\end{equation}

We now replace the dependent variables $X_{12}$ and $X_{21}$ by independent variables $Y_{12}$ and $Y_{21}$ such that
${\displaystyle P(Y_{12} = u) = p_u}$ and ${\displaystyle P(Y_{21} = v) = p_{m-v}}$ for $0\leq u,v\leq m,$
but we do not require that  $Y_{12} + Y_{21} = m$.
This gives a new joint distribution function
\begin{equation}
\label{eq:M3}
F_1:=F_1\left(k_1,\ldots, k_n\right)=P(s^{'}_1 + Y_{12} \leq k_1, s^{'}_2 + Y_{21}\leq k_2, s_3\leq k_3,\ldots, s_n\leq k_n).
\end{equation}
For notational convenience we shall temporarily suppress  the $s_i\leq k_i$ terms for $3\leq i\leq n$ in relations \eqref{eq:M2} and \eqref{eq:M3}
and in what follows. When we subtract $F$ from $F_1$ and then sum over the possible values of $Y_{12}, Y_{21}, X_{12}$, and
$X_{21} = m - X_{12}$, bearing in mind that these variables are independent of the other variables, we find that

\begin{align}
\label{eq:M4}
&
F_1 - F=P(s^{'}_1\leq k_1 - Y_{12}, s^{'}_2\leq k_2 - Y_{21}) - P( s^{'}_1\leq k _1 - X_{12}, s^{'}_2\leq k_2 - X_{21}) \nonumber\\
&
=
\sum_{u=0}^{m} \sum_{v=0}^{m}
P(s^{'}_1\leq k_1 - u, s^{'}_2\leq k_2 - v) \left\{P(Y_{12} = u, Y_{21} = v) - P(X_{12} = u, X_{21} = v)\right\}\nonumber\\
&
=\sum_{u=0}^{\min{(m, k_1)}} \sum_{v=0}^{\min{(m, k_2)}}
P(s^{'}_1\leq k_1 - u, s^{'}_2\leq k_2 - v) p(u,v),
\end{align}
where
\begin{equation}
\label{eq:M6}
p(u,v)=P(Y_{12} = u, Y_{21} = v) -  P(X_{12} = u, X_{21} = v)=
\left\{
\begin{array}{ccc}
  p_u p_{m-v} & if & u+v \neq m \\
  p_u^2 -  p_u & if & u+v=m,
\end{array}
\right.
\end{equation}
and where we have appealed to assumption $\eqref{eq:L}$ at the last step.

We want to show that
\begin{equation}
F_1 - F \geq 0.
\end{equation}

To establish this, we need to introduce some more notation and Assertions 1 and 2 (with proofs in Appendix A and B) in order to obtain a simpler form of relation \eqref{eq:M4}.

We let
\begin{equation}
\label{eq:M8}
R(g, h)=P( s^{'}_1=k_1 -g,  s^{'}_2=k_2-h)
\end{equation}
and
\begin{equation}
\label{eq:M9}
W(g,h)=\sum_{u=0}^{\min(m,g)}\sum_{v=0}^{\min(m,h)}p(u,v)
\end{equation}
for  $0\leq g\leq k_1$   and  $0\leq h \leq k_2$.
\bigskip

\noindent
{\bf Assertion 1}.
${\displaystyle
F_1-F=\sum_{g=0}^{k_1}\sum_{h=0}^{k_2} R(g,h)W(g,h).
}$
\medskip

\noindent
{\bf Assertion 2}.
If $0\leq g \leq k_1$ and  $0\leq h\leq k_2$, then
                                           $W(g,h)\geq 0$.
\bigskip

Since $R(g,h)$ and $W(g,h)$ are each nonnegative, by definition and by Assertion $2$, it follows from the relation in
Assertion $1$ that  $F_1 - F\geq 0$, as required.
To complete the proof of Theorem \ref{eq:Lemma2}, we proceed as follows. Suppose the pairs $\left\{{(i,j): 1\leq i < j \leq n}\right\}$ are
lexicographically ordered and labelled from 1 to $n(n - 1)/2$. We defined $F_1$ as the distribution function obtained from
the distribution function $F$ by replacing the dependent variables $X_{12}$ and $X_{21}$ by the independent variables $Y_{12}$
and $Y_{21}$; where  $Y_{12}$ and $Y_{21}$ have the same distribution as $X_{12}$ and $X_{21}$  except that we do not require that
$Y_{12}  + Y_{21} = m$.  Similarly, if $1 < t\leq n(n-1)/2$ and the pair $(i,j)$ has label $t$, then $F_t$ is defined to be the distribution function
obtained from the function $F_{t - 1}$ by replacing the dependent variables $X_{ij}$ and $X_{ji}$ by the independent variables
$Y_{ij}$ and $Y_{ji}$; where  $Y_{ij}$ and $Y_{ji}$ have the same distribution as $X_{ij}$ and $X_{ji}$ except that we do not require that
$Y_{ij} + Y_{ji} = m$. The conclusion that $F_{t -1}\leq F_t$ for $1 < t\leq n(n-1)/2$ follows by essentially the same type of argument as
was used to show that $F\leq F_1$. When we combine these inequalities, we find that
\begin{align*}
&
F(k_1,\ldots,k_n)=F\leq F_1\leq F_2\leq \cdots \leq F_{n(n-1)/2}\\
&
= P\left(\sum_{j \ne 1}Y_{1j}\leq k_1,\ldots,\sum_{j \ne n}Y_{nj}\leq k_n\right)= P\left(\sum_{j \ne 1}Y_{1j}\leq k_1\right)\cdots P\left(\sum_{j \ne n}Y_{nj}\leq k_n\right)\\
&
= P\left(\sum_{j \ne 1}X_{1j}\leq k_1\right)\cdots P\left(\sum_{j \ne n}X_{nj}\leq k_n\right)=P(s_1\leq k_1)\cdots P(s_n\leq k_n),
\end{align*}
since the variables $Y_{ij}$ are independent and the variables $\sum_{j \ne i}Y_{ij}$ and $\sum_{j \ne i}X_{ij}$ have the same distribution for
each $i$. This completes the proof of the Theorem \ref{eq:Lemma2}.
\end{proof}

Put $\sigma^2_i(n-1)=\sum_{j=1, j\ne i}^{n}Var(X_{ij})$ for $i=1,\ldots,n$.
\begin{thm}
\label{eq:HH}
If the probabilities associated with a generalized tournament satisfy condition \eqref{eq:L} and, for $i=1,\ldots,n$,
$\sigma_i(n-1)\rightarrow \infty$ as $n\rightarrow \infty$, then $s_{(n)}^{*}-\sqrt{2\log(n-1)}\rightarrow 0$ in probability.
\end{thm}

\begin{proof}
The proof differs from the arguments used by \cite{H1963} to establish Result 1 in two ways. (i) Instead of using the large deviation result for Bernoulli random variables (stated on p.193 of the 3rd edition of \cite{F1968}), we use Cram\'{e}r-type large deviation results for independent non-identically distributed random variables, as stated below; (ii) we use our Theorem \ref{eq:Lemma}, a stronger form of the Lemma 1 that Huber used.

For each fixed $i$ the random variables $X_{ij}, j\neq i$, are independent, but not necessarily identically distributed, such that $0 \leq X_{ij}\leq m$ and $m$ is finite. Hence, it follows from \cite{F1971}(p. 553, Theorem 3) that
\begin{equation}
\label{eq:F}
P(s_i^{*}>x_{n-1})\sim 1-\Phi(x_{n-1}),
\end{equation}
provided that ${\displaystyle x_{n-1}=o\left((\sigma_i(n-1))^{1/3}\right)}$,
${\displaystyle \frac{x_{n-1}^{3}}{\sigma_i(n)}\rightarrow 0}$ and $\sigma_i(n-1)\rightarrow \infty$, as $n\rightarrow \infty$,
where $\Phi()$ is the CDF of a standard normal variable.

It is well known (see for example, Lemma 2, p. 175 \cite{F1968}) that as $x_{n-1}\rightarrow \infty$,
\begin{equation}
\label{eq:MR}
1-\Phi(x_{n-1})\sim \frac{1}{x_{n-1}}\varphi(x_{n-1}),
\end{equation}
where $\varphi()$ is the PDF of a standard normal random variable.

Let $\epsilon>0$ and put
\begin{equation}
\label{eq:Ch}
x_{n-1}^{\pm}=[2\log(n-1)-(1 \pm \epsilon)\log(\log(n-1))]^{1/2}.
\end{equation}
Combining \eqref{eq:F} and \eqref{eq:MR}, we obtain
\begin{equation*}
P(s_i^{*}>x_{n-1}^{\pm})\sim \frac{(\log(n-1))^{\pm \epsilon/2}}{\sqrt{4\pi}(n-1).}
\end{equation*}
Then, choosing $c^{''}$ and $c^{'}$ such that $c^{''}<\frac{1}{\sqrt{4\pi}}<c^{'}$,
we find that for all sufficiently large $n$
\begin{equation}
\label{eq:LHS}
P(s_{(n)}^{*}>x_{n-1}^{-})\leq \sum_{i=1}^{n}P(s_{i}^{*}>x_{n-1}^{-})<n c^{'}\frac{(\log(n-1))^{-\epsilon/2}}{n-1}.
\end{equation}

Using Theorem \ref{eq:Lemma2} we find that for sufficient large $n$
\begin{align}
\label{eq:RHS}
&
P(s_{(n)}^{*}\leq x_{n-1}^{+})\leq \prod_{i=1}^{n}P(s_{i}^{*}\leq x_{n-1}^{+})<
\left[1-c^{''}\frac{(\log(n-1))^{\epsilon/2}}{n-1}\right]^n\leq e^{-c^{''}n\frac{(\log(n-1))^{\epsilon/2}}{n-1}}.
\end{align}
Theorem \ref{eq:HH} now follows from \eqref{eq:LHS} and \eqref{eq:RHS}.
\end{proof}

\begin{com}
It is not difficult to see that the proofs of Theorem \ref{eq:Lemma2} and \ref{eq:HH}  still go through in the following more general situations:   when
\begin{enumerate}
\item[(a)] the probabilities that appear in condition \eqref{eq:L} can take on  different values  $(p_u)^{(i,j)}$ for different ordered pairs (i, j); and
\item[(b)]the integer $m$ that appears in condition \eqref{eq:L} and the relation $X_{ij}+X_{ji}=m$  can take on  different values  $m_{ij}=m_{ji}$ for different unordered pairs $(i, j)$.
\end{enumerate}
\end{com}

\begin{cor}
\label{eq:cor1}
If $E(s_1)=\cdots=E(s_n)$, $Var(s_1)=\cdots=Var(s_n)$, then under the conditions of Theorem \ref{eq:HH},
we have $(s_{(n)}-E(s_1))/\sqrt{Var(s_1)}-\sqrt{2\log(n-1)}\rightarrow 0$ in probability as $n\rightarrow \infty$.
\end{cor}

\begin{com}
The assumptions in Corollary \ref{eq:cor1} imply that in our model
$E(s_i)=E(s_j)=\frac{1}{2}\text{(maximum total score possible)}$
and $Var(s_i)=Var(s_j)$; whereas in Huber's model, with no outlier, $E(X_{ij})=1/2, Var(X_{ij})=1/4$ for all $i \neq j$.
\end{com}

\section{Examples}
We present a few examples where the assumptions of Corollary \ref{eq:cor1} are satisfied.
\begin{ex}[Uniform distribution] If $p_{u}=\frac{1}{m+1}$, for $u=0,1,\ldots,m$, and if $n\rightarrow \infty$, then
$$s_{(n)}-\left\{\frac{(n-1)m}{2}+\sqrt{\frac{(n-1)(\log(n-1)){m(m+2)}}{6}}\right\}\rightarrow 0$$ in probability.
\end{ex}

\begin{ex}[Symmetric Binomial distribution] If $X_{ij}\sim Bin (m, p=1/2)$, and if $n\rightarrow \infty$, then
$$s_{(n)}-\left\{\frac{(n-1)m}{2}+\sqrt{\frac{(n-1)(\log(n-1)){m}}{2}}\right\}\rightarrow 0$$ in probability.
\end{ex}

\begin{ex}[Chess round-robin tournament with draws]
\label{eq:Chess}
In this case $X_{ij}$, the score of player $i$ after the game with player $j, j\neq i$,
equals $1, 1/2$, or $0$ accordingly as player $i$ wins, draws, or loses the game against player $j$.
Therefore, ${\displaystyle X_{ij}+X_{ji}=1}, i\neq j$. For any $i\neq j$ let $p=P\left(X_{ij}=1/2\right)$,
and assume $P\left(X_{ij}=1\right)=P\left(X_{ij}=0\right)$.
In this case we find that  $E(s_1)=\frac{n-1}{2}, Var(s_1)=\frac{(n-1)(1-p)}{4}$, and
if $n\rightarrow \infty$, then $$s_{(n)}-\left\{\frac{n-1}{2}+\sqrt{\frac{(n-1)(\log(n-1)){(1-p)}}{2}}\right\}\rightarrow 0$$ in probability.
\end{ex}

\begin{ex}[Non-identically distributed scores]
Suppose $n = 2k+1$ and let $p_1,\ldots,p_k$ denote arbitrary probabilities (not equal to $0$ or $1$). Suppose points labelled $1, 2,\ldots, n$ (corresponding to the $n$ players) are arranged around the circumference of a circle in that order--so it makes sense to talk of one point being the successor or predecessor of another point etc.
Consider three vertices labelled $h$, $i$, and $j$ where $1\leq h, i, j \leq n$ and $h + d = i$ and $i + d = j$ for some $d$ such that $1\leq d\leq k$; we reduce labels modulo $n$ when necessary. For each such triple, let
\begin{align}
&
P(X_{hi}=u)= P(X_{ij} = u) = C(m, u)(p_d)^{u}( 1 - p_d)^{m - u},\nonumber
\\
&
P(X_{ih} = u) = P(X_{ji} = u) = C(m, u)( 1 - p_d)^{ u}  (p_d)^{m - u},
\end{align}
where $C(m, u)$ denotes the binomial coefficient $m$-choose-$u$.

Then
\begin{align}
&
E(X_{ij}) + E(X_{ih})=mp_d + m(1 - p_d)=m,\,\,\,Var(X_{ij}) + Var(X_{ih})=2mp_d(1 - p_d).
\end{align}

Consequently,
\begin{align}
&
E(s_i) = \sum_{j \ne i} E(X_{ij}) = m(n - 1)/2,  Var(s_i) = \sum_{j \ne i}Var(X_{ij}) = 2m( v_1 +\cdots+ v_k)
\end{align}
for all $i$, where $v_d = p_d(1 - p_d), d=1,\ldots,k$;
and if $n\rightarrow \infty$, then $$s_{(n)}-\left\{\frac{m(n-1)}{2}+2\sqrt{(\log(n-1)m(v_1 +\cdots+ v_k)}\right\}\rightarrow 0$$ in probability.

When $n = 2k$, we proceed essentially as before as far as the probabilities  $p_d$  are concerned, for $1\leq d\leq k - 1$. For the remaining case, consider diametrically opposite points  $i$ and $j$, where $1\leq i\leq k$ and  $k + 1\leq j\leq n$:
if $j = i + k$, let  $P(X_{ij}=u) = C(m, u)(1/2)^m$; and if  $i= j + k$ let $P(X_{ji} = u) = C(m, u)(1/2)^m$.
Consequently,
\begin{align}
&
E(s_i) = m(k - 1 + 1/2) = m(n - 1)/2,\,\,\, Var(s_i) = m(2(v_1 + ....+v_{k - 1})+ 1/4)
\end{align}
for all $i$;
and if $n\rightarrow \infty$, then $$s_{(n)}-\left\{\frac{m(n-1)}{2}+2\sqrt{(\log(n-1)m(v_1 +\cdots+ v_{k-1}+1/8)}\right\}\rightarrow 0$$ in probability.
\end{ex}

\begin{ex}[Non-identically distributed scores  with two different values of $m$, $m_w$ and $m_b$]
Let  $n = 3k$ and suppose the $n$ competitors are split into three classes (1), (2), and (3) of $k$ competitors each.
Let $p$ and $q$ denote two constants such that $0 < q < p < 1$. Let $i$ and $j$ denote any two competitors.
If $i$ and $j$ belong to the same class -(1), (2), or (3)-- and $i < j$,  let
$P(X_{ij} = u) = C(m_{w}, u)(1/2)^{m_{w}} = P(X_{ji}=m_{w} - u), u=0,1,\ldots, m_{w}$
If $i$ belongs to class (h)  and $j$ belongs to class h+1, where $1 \leq h  \leq 3$ and $h + 1$ is reduced modulo $3$, if necessary, then $P(X_{ij} = u) = C(m_{b}, u)(p)^{u}(q)^{m_{b} - u} =P(X_{ji} = m_{b} - u), u=0,1,\ldots,m_{b}$.
In this case we find that
\begin{align}
&
E(s_i) = (k - 1)m_{w}/2 + km_{b},\,\,
Var(s_i) = (k-1)m_{w}/4+2km_{b}pq;
\end{align}
and if $n\rightarrow \infty$, then $$s_{(n)}-\left\{(k - 1)m_{w}/2 + km_{b}+\sqrt{2\log(n-1)\left((k-1)m_{w}/4+2km_{b}pq\right)}\right\}\rightarrow 0$$ in probability.
\end{ex}

\begin{remark}
Suppose the outcomes of the competitions depend primarily upon three attributes of the participants: strength, speed,  and experience. and suppose that participants in classes (1), (2), and (3) excel in strength, speed, and experience, respectively. Our assumptions are that strong players have an advantage when competing against fast players but are at a disadvantage when competing against more experienced players, and similarly for the other combinations. And that players that excel in the same attribute are equally likely to win when competing against each other. Thus the situation here is somewhat similar to that arising in the scissors, paper, stone  game except that we have introduced probabilities instead of certainties  for the outcomes here.
\end{remark}

\begin{ex}
If $m = 2k$, let $p_u = p_{m - u} = (u + 1)/(k + 1)^2$ for $u = 0,1,\ldots,k$.
Then $E(X_{ij}) = m/2$ and $Var(X_{ij})=m(m+4)/24$ for all distinct $i$ and $j$.
Hence $E(s_i) = m(n - 1)/2$ and $Var(s_i)=(n - 1)m(m+4)/24$ for all $i$;  and if $n\rightarrow \infty$, then
$$s_{(n)}-\left\{\frac{(n-1)m}{2}+\sqrt{\frac{(n-1)(\log(n-1)){m}{(m+4)}}{12}}\right\}\rightarrow 0$$ in probability.
\end{ex}

\begin{ex}
If $m = 4$,  let $p_0=p_4=L^3$,  $p_1=p_3=L^2$,  and $p_2 = L$  where  $L=.4406197\ldots$ is the positive root of the equation $L + 2L^2 + 2L^3 = 1.$
Then $E(X_{ij})=L^2 +2L+3L^2 + 4L^3=4(L/2 + L^2 + L^3)=4/2=2$  and  $Var(X_{ij}) = 2(L^2 + 4L^3) = 1.0726468\ldots$
for all distinct  i and j. Hence, $E(s_i) = 2(n - 1)$  and $Var(s_i))=1.0726468(n - 1)$ for all $i$; and if $n\rightarrow \infty$, then
$$s_{(n)}-\left\{2(n-1)+\sqrt{2.1452936(n-1)(\log(n-1))}\right\}\rightarrow 0$$ in probability.
\end{ex}

\section*{Acknowledgements}
We would like to thank the Editor, Associate Editor and referee for the insightful and helpful comments that led to significant improvements in the paper.
YM thanks Abram Kagan for describing a score issue in chess round-robin tournaments with draws.
The research of YM was supported by grant no. 2020063 from the
United States--Israel Binational Science Foundation (BSF), Jerusalem, Israel.

\section*{Appendix}
\appendix
\section{Proof of Assertion 1}
\begin{proof}
When we apply definitions \eqref{eq:M8}, and \eqref{eq:M9} to the last expression in relation \eqref{eq:M4}, we find that
\begin{align}
&
F_1-F=\sum_{u=0}^{\min{(m, k_1)}} \sum_{v=0}^{\min{(m, k_2)}} p(u,v) P(s^{'}_1\leq k_1 - u, s^{'}_2\leq k_2 - v)\nonumber\\
&
=\sum_{u=0}^{\min{(m, k_1)}} \sum_{v=0}^{\min{(m, k_2)}} p(u,v) \sum_{g=u}^{k_1} \sum_{h=v}^{k_2} R(g,h)\nonumber\\
&
=\sum_{g=0}^{k_1} \sum_{h=0}^{k_2} R(g,h)\sum_{u=0}^{\min(m,g)}\sum_{v=0}^{\min(m,h)}p(u,v)\nonumber\\
&
=\sum_{g=0}^{k_1} \sum_{h=0}^{k_2} R(g,h) W(g,h),
\end{align}
as required.
\end{proof}

\section{Proof of Assertion 2}
\begin{proof}
Let
\begin{equation}
Q_u = p_0 +\ldots+ p_u
\end{equation}
for  $0\leq u\leq m$, where we adopt the convention that $Q_{-1} = 0.$
\smallskip

We consider various cases separately.
\smallskip

\noindent
$Case (i)$. $0\leq g, h\leq m$  and  $0\leq g + h < m$.

If  $0\leq u\leq g$  and $0\leq v \leq h$, then $u + v \leq g + h < m$, so $p(u,v) = p_up_{m-v}$ by relation \eqref{eq:M6}. Consequently,
$$W(g, h) = \sum_{u=0}^g \sum_{v=0}^h p_u p_{m-v}=(p_0 +\cdots+p_g)(p_m +p_{m-1}+\cdots+p_{m-h}) = Q_g(1-Q_{m-h-1})\geq 0.$$

\noindent
$Case (ii).$ $0\leq g, h \leq m$  and  $g + h\geq m.$

In this case it follows from relation \eqref{eq:M6} that
                                   $$W(g,h) = Q_g(1-Q_{m-h-1}) - {\sum}' p_t$$

\noindent
where the sum is, in effect, over all pairs $(t, m - t)$ such that $m - h \leq t \leq g$; hence,

\begin{align*}
&
W(g, h) =Q_g(1-Q_{m-h-1}) - (p_{m - h} +\cdots+ p_g)\\
&
= Q_g(1-Q_{m-h-1})-(Q_g-Q_{m-h-1})=Q_{m-h-1}(1 - Q_g)\geq 0.
\end{align*}

Notice, in particular, that
\begin{equation}
\label{eq:M13}
W(m,h) = 0\,\,\,\,\,\text{for}\,\,\,\,\, 0\leq h\leq m\,\,\,\,\, \text{and}\,\,\,\,\,\, W(g, m)=0\,\,\,\,\,\, \text{for}\,\,\,\,\,
0\leq g \leq m.
\end{equation}

We also point out that if $0\leq g,h<m$, then $W(g, h)=W(m-1-h, m-1-g)$.
\smallskip

\noindent
$Case (iii).$ $0\leq h\leq m<g\leq k_1$, or $0\leq g\leq m < h\leq k_2$, or  $m < g\leq k_1$ and  $m < h\leq k_2$.

If  $0\leq h \leq m < g \leq k_1$, then it follows from \eqref{eq:M9} and \eqref{eq:M13} that

                 $$W(g,h) = \sum_{u=0}^m \sum_{v=0}^h p(u,v) = W(m,h) = 0.$$
\noindent
If  $0\leq g\leq m < h\leq k_2$, then it follows from \eqref{eq:M9} and \eqref{eq:M13} that

                 $$W(g,h) = \sum_{u=0}^g \sum_{v=0}^m p(u,v) = W(g,m) = 0.$$
\noindent
If  $m < g\leq k_1$ and  $m < h\leq k_2$, then it follows from \eqref{eq:M9} and \eqref{eq:M13} that

                 $$W(g,h) = \sum_{u=0}^m \sum_{v=0}^m p(u,v) = W(m,m) = 0.$$
\end{proof}

{}

\end{document}